\documentclass[11pt]{amsart}

\usepackage{hyperref}
\hypersetup{nesting=true,debug=true,naturalnames=true}
\usepackage{graphicx,amssymb,upref}
\usepackage[usenames,dvipsnames]{pstricks}
\usepackage{epsfig}
\usepackage{color}
\usepackage[latin1]{inputenc}
\usepackage[T1]{fontenc}
\usepackage{amsmath, amsthm}
\usepackage[english]{babel}
\usepackage{amssymb}
\usepackage{latexsym}
\usepackage{graphicx,amssymb,upref}
\usepackage[all]{xy}
\hypersetup{nesting=true,debug=true,naturalnames=true}
\usepackage{tikz}
\usetikzlibrary{matrix}

\date{\today}

\newtheorem{theorem}{Theorem}[section]

\theoremstyle{definition}
\newtheorem{definition}[theorem]{Definition}
\newtheorem{example}[theorem]{Example}
\newtheorem{examples}[theorem]{Examples}

\newcommand{\ot}{\otimes}
\newcommand{\co}{\circ}

\let\<\langle
\let\>\rangle

\let\uml\"

\title[Factorizations of Hopf quasigroups]{Factorizations of Hopf quasigroups}
\title{Factorizations of Hopf quasigroups} 

\begin{document}

\maketitle

\begin{center}
{\bf Ram\'on
Gonz\'{a}lez Rodr\'{\i}guez$^{a,b}$}.
\end{center}

\vspace{0.1cm}
\begin{center}
{\small \vspace{0.4cm}  [https://orcid.org/0000-0003-3061-6685]}
\end{center}
\begin{center}	{\small $^{a}$ CITMAga, 15782 Santiago de Compostela, Spain}
\end{center}
\begin{center}
{\small $^{b}$ Departamento de Matem\'{a}tica Aplicada II, Universidade de Vigo,  E-36310 
Vigo, Spain\\email: rgon@dma.uvigo.es}
\end{center}

\vspace{0.1cm}

\begin{abstract}  
In this paper  we introduce the notion of factorization in the Hopf quasigroup setting and we prove that, if $A$ and $H$ are Hopf quasigroups such that their antipodes are isomorphisms, a Hopf quasigroup $X$ admits a factorization as $X=AH$ iff $X$ is isomorphic  to a double cross product $A\bowtie H$ as Hopf quasigroups.
\end{abstract} 

\vspace{0.2cm} 



\section{Introduction}  
Let ${\mathbb F}$ be a field and denote by $\otimes$ the tensor product in the category of vector spaces over ${\mathbb F}$ denoted by ${\mathbb F}$-{\sf Vect}. The double cross product of two Hopf algebras $A$ and $H$ in ${\mathbb F}$-{\sf Vect} was introduced by Majid in \cite[Proposition 3.12]{MAJ} (see also \cite[Theorem 7.2.2]{MAJDCP}) as a new Hopf algebra structure defined in the tensor product $A\otimes H$ and determined by a matched pair $(A,H)$. A matched pair  of Hopf algebras is a system $(A,H)$, where $A$ and $H$ are Hopf algebras, $A$ is a left $H$-module coalgebra with action $\varphi_{A}:H\otimes A\rightarrow A$, $H$ is a right $A$-module coalgebra with action $\phi_{H}:H\otimes A\rightarrow H$ and some suitable compatibility conditions hold for all $h,g\in H$ and $a,b\in A$. Using the Heyneman-Sweedler's convention and the notations $\varphi_{A}(h\otimes a)=h \triangleright a$, $\phi_{H}(h\otimes a)=h\triangleleft a$, these conditions can be written as follows:
$$h\triangleright 1_{A}=\varepsilon(h)1_{A}, \;\; h\triangleright (ab)=(h_{(1)}\triangleright a_{(1)})((h_{(2)} \triangleleft a_{(2)})\triangleright b),$$
$$1_{H}\triangleleft a=\varepsilon(a)1_{H}, \;\; (hg)\triangleleft a =(h \triangleleft  (g_{(1)} \triangleright a_{(1)})) (g_{(2)}\triangleleft a_{(2)}), $$
$$h_{(1)}\triangleleft a_{(1)}\otimes h_{(2)}\triangleright a_{(2)}=h_{(2)}\triangleleft a_{(2)}\otimes h_{(1)}\triangleright a_{(1)}.$$

If $(A,H)$ is a matched pair  of Hopf algebras, the double cross product $A\bowtie H$ of $A$ with $H$ is the Hopf algebra built on the vector space $A\otimes H$ with product 
$$(a\otimes h)(b\otimes g)=a(h_{(1)} \triangleright b_{(1)})\otimes (h_{(2)} \triangleleft b_{(2)})g$$
and tensor product unit, counit, coproduct and antipode 
$$\lambda_{A\bowtie H}(a\otimes h)=\lambda_{H}(h_{(2)})\triangleright \lambda_{A}(a_{(2)})\otimes \lambda_{H}(h_{(1)})\triangleleft \lambda_{A}(a_{(1)}).$$
where $\lambda_{H}$ is the antipode of $H$ and $\lambda_{A}$ is the antipode of $A$.

Following \cite{MAJDCP}, a Hopf algebra $X$ factorises as $X=AH$ if there exists sub-Hopf algebras $A$ and $H$ with inclusion maps $i_{A}$ and $i_{H}$ such that the map $\omega (a\otimes h)=i_{A}(a)i_{H}(h)$ is an isomorphism of vector spaces. As was proved by Majid in \cite[Theorem 7.2.3]{MAJDCP}, $X$ factorises as $X=AH$ iff there exists a matched pair of Hopf algebras $(A,H)$ such that $X$ is isomorphic to $A\bowtie H$ as Hopf algebras. 

On the other hand, the theory of distributive laws between monads was initiated by Beck \cite{Beck} and Barr \cite{Barr} in the seventies of the last century. A distributive law between two algebras $A$ and $H$ is a morphism $\Psi:H\ot A\rightarrow A\otimes H$ which is compatible with the algebra structures. It is well-known that a distributive law $\Psi:H\ot A\rightarrow A\otimes H$ induces  an algebra structure on the tensor product $A\otimes H$  commuting with  the action associated to the product of $A$ on the left and with  the action associated to the product of $H$ on the right. This algebra, denoted by  $A\ot_{\Psi}H$, is the wreath product of $A$ and $H$ and its unit and product are defined by $1_{A\ot_{\Psi}H}=1_{A}\ot 1_{H}$ and $$\mu_{A\ot_{\Psi}H}=(\mu_{A}\ot \mu_{H})\co (id_{A}\ot \Psi\ot id_{H}),$$ where  $\mu_{A}$, $\mu_{H}$ are the corresponding products and $id_{A}$, $id_{H}$ are the identity morphisms for $A$ and $H$ respectively. It is well known that  bialgebras are algebras in the category of coalgebras. From this point of view, a distributive law in the category of bialgebras is a distributive law between the underlying algebras satisfying that  is a coalgebra morphism. This kind of  distributive laws induce a wreath product bialgebra, where the product is the wreath product and the colagebra structure is the one associated to the tensor product coalgebra. A relevant example of these wreath products are the double crossed product quoted in the previous page where the distributuve law is $$\Psi(h\otimes a)=h_{(1)} \triangleright b_{(1)}\otimes h_{(2)} \triangleleft b_{(2)}.$$

In the literature we can find similar constructions of double cross products. For example, in the associative case are relevant the double cross products associated to  matched pairs of groups \cite{T} (i.e. Hopf algebras in the category  of sets) and the double cross products associated to  matched pairs of groupoids \cite{Mar}. In \cite{our2} an extension of this kind of products was presented in a non-associative setting for matched pairs of Hopf quasigroups  as a generalization of the results proved in \cite{MajidCMUC} for quasigroups. The notion of Hopf quasigroup in ${\mathbb F}$-{\sf Vect} was introduced  by Klim and Majid in \cite{Majidesfera} and it is a particular case  of unital coassociative $H$-bialgebra (see \cite{PI07}) and also of a quantum quasigroup (see \cite{SM1} and \cite{SM2}). This non-associative generalization of Hopf algebras include as a particular cases the loop algebra for a loop $L$ with the inverse property (see \cite{Majidesfera}, \cite{Bruck}) and also  the enveloping algebra $U(M)$ of a Malcev algebra over ${\mathbb F}$ (see \cite{Majidesfera} and \cite{PIS}). Several articles have been published in recent years devoted to  the study  of some kind of products between Hopf quasigroups as for example: \cite{BrzezJiao1} and \cite{BrzezJiao2} for the theory of smash products of Hopf quasigroups; \cite{FT1} and  \cite{FT} for the theory of  twisted smash products of Hopf quasigroups; \cite{FT} for Hopf quasigroups obtained by the twist double method; \cite{our2}  for Hopf quasigroups associated to skew pairings  and Hopf quasigroups obtained as double cross products of Hopf quasigroups. As was pointed in \cite{RGR}, in all these cases the product is determined by  a morphism $\Psi$  satisfying some conditions that are close to the ones involved in the classical definition of distributive law. Taking this into consideration, in \cite{RGR} the author introduce a  notion of "non-associative" distributive law, called $a$-comonoidal distributive law,   that permits to understand the products cited in the previous lines  with a general point of view. In the final section of \cite{RGR} we can find the proof of the following fact: For two Hopf quasigropups $A$ and $H$, the tensor product $A\ot H$ with the corresponding  wreath product, i.e., the  product associated to an $a$-comonoidal distributive law, becomes a  Hopf quasigroup, where the coalgebra structure is the one of the tensor product coalgebra. 

Taking into account what was said in the previous paragraphs, it is natural to ask when a Hopf quasigroup $X$ admits a factorization. To answer this question is the main motivation of this paper. 

The structure of the paper is as follows. In Section 2 we recall some necessary background about Hopf (co)quasigroups in a monoidal setting and in Section 3 we present the main facts of the theory of wreath products associated to an $a$-comonoidal distributive law. In Section 4 we introduce the notion of factorization for Hopf quasigrous and we prove the main theorem of this paper that asserts the following: Let $H$, $A$, $X$ be Hopf quasigroups such that the antipodes of $H$ and $A$ are isomorphisms. Then, $X$ factorizes as $X=AH$ iff there exists a matched pair of Hopf quasigroups $(A,H)$ such that $X$ is isomorphic to the double cross product $A\bowtie H$ as Hopf quasigroups. Also, in this last section, we discuss an example of a non-commutative, non-cocommutative Hopf quasigroup constructed as a double cross product.  Finally, by dualisation, in this paper we show that we can obtain similar results for Hopf coquasigroups.

\section{Preliminaries } 

From now on {\sf C} denotes a strict symmetric monoidal category with tensor product $\ot$, unit object $K$ and natural isomorphism of symmetry $c$. Recall that  a monoidal category is a category ${\sf C}$ equipped with a tensor product functor $\ot :{\sf C}\times {\sf C}\rightarrow {\sf C}$,  a unit object  $K$ of ${\sf C}$ and  a family of natural isomorphisms $a_{M,N,P}:(M\ot N)\ot P\rightarrow M\ot (N\ot P),$ $r_{M}:M\ot K\rightarrow M$,  $l_{M}:K\ot M\rightarrow M,$ in ${\sf C}$ (called  associativity, right unit and left unit constraints, respectively) satisfying the Pentagon Axiom and the Triangle Axiom, i.e.,
$$a_{M,N, P\ot Q}\co a_{M\ot N,P,Q}= (id_{M}\ot a_{N,P,Q})\co a_{M,N\ot P,Q}\co (a_{M,N,P}\ot id_{Q}),$$
$$(id_{M}\ot l_{N})\co a_{M,K,N}=r_{M}\ot id_{N},$$
where $id_{X}$ denotes the identity morphism for each object $X$ in ${\sf C}$. A monoidal category is called strict if the associativity, right unit and left unit constraints are identities.  A strict monoidal category ${\sf C}$ is symmetric  if it has  a  family of natural isomorphisms $c_{M,N}:M\ot N\rightarrow N\ot M$ such that the equalities
$$
c_{M,N\ot P}= (id_{N}\ot c_{M,P})\co (c_{M,N}\ot id_{P}),\;\;
c_{M\ot N, P}= (t_{M,P}\ot id_{N})\co (id_{M}\ot c_{N,P}),\;\; 
$$
$$c_{N,M}\co c_{M,N}=id_{M\ot N},$$
hold for all $M$, $N$ in ${\sf C}$.

Considering that it is well known that every non-strict monoidal category is monoidal equivalent to a strict one, we can assume without loss of generality that the category {\sf C} is strict and then we omit  explicitly  the associativity and unit constraints. Thus, the  results proved  in this paper for objects and morphisms in {\sf C} remain valid for every non-strict symmetric monoidal category, what would include for example the category ${\mathbb F}$-{\sf Vect} of vector spaces over a field ${\mathbb F}$,  the category $R$-{\sf Mod} of left modules  over a commutative ring $R$, or the category {\sf Set} of sets.  In what follows, for simplicity of notation, given objects $M$, $N$, $P$ in ${\mathcal
	C}$ and a morphism $f:M\rightarrow N$, we write $P\ot f$ for
$id_{P}\ot f$ and $f \ot P$ for $f\ot id_{P}$.

A magma in ${\sf C}$ is a pair $A=(A,  
\mu_{A})$, where  $A$ is an object in ${\sf C}$ and $\mu_{A}:A\otimes A\rightarrow A$ (product) is a morphism in ${\sf C}$.
A unital magma  in ${\sf C}$ is a triple $A=(A, \eta_A, 
\mu_{A})$, where $(A,  
\mu_{A})$ is a magma in ${\sf C}$ and
$\eta_{A}: K\rightarrow A$ (unit) is a morphism in ${\sf C}$ such that 
$\mu_{A}\circ (A\otimes \eta_{A})=id_{A}=\mu_{A}\circ
(\eta_{A}\otimes A)$. A monoid in ${\sf C}$ is a unital magma $A=(A, \eta_A, \mu_{A})$ in ${\sf C}$ satisfying   $\mu_{A}\circ (A\otimes
\mu_{A})=\mu_{A}\circ (\mu_{A}\otimes A)$, i.e., the product $\mu_{A}$ is associative. Given two unital magmas (monoids) $A$ and $B$, a morphism $f:A\rightarrow B$ in ${\sf C}$ is called a morphism of unital magmas (monoids) if  $f\circ \eta_{A}= \eta_{B}$ (i.e., the morphism $f$ is unitary) and $\mu_{B}\circ (f\otimes f)=f\circ \mu_{A}$ (i.e., the morphism $f$ is multiplicative). 

Also,
if $A$, $B$ are unital magmas (monoids) in ${\sf C}$, the object $A\otimes
B$ is a unital magma (monoid) in
${\sf C}$, where $\eta_{A\otimes B}=\eta_{A}\otimes \eta_{B}$
and $\mu_{A\otimes B}=(\mu_{A}\otimes \mu_{B})\circ (A\otimes c_{B,A}\otimes B)$. If $A=(A, \eta_A,  \mu_{A})$ is a unital magma so is $A^{op}=(A, \eta_A,  \mu_{A}\co c_{A,A})$.

A comagma in ${\sf C}$ is a pair ${D} = (D, \delta_{D})$, where $D$ is an object in ${\sf C}$ and $\delta_{D}:D\rightarrow D\otimes D$ (coproduct) is a morphism in
${\sf C}$. A counital comagma in ${\sf C}$ is a triple ${D} = (D,
\varepsilon_{D}, \delta_{D})$, where $(D, \delta_{D})$ is a comagma in ${\mathcal
	C}$ and $\varepsilon_{D}: D\rightarrow K$ (counit) is a  morphism in
${\sf C}$ such that $(\varepsilon_{D}\otimes D)\circ
\delta_{D}= id_{D}=(D\otimes \varepsilon_{D})\circ \delta_{D}$.  A comonoid  in ${\sf C}$ is a counital comagma in ${\sf C}$ satisfying  $(\delta_{D}\otimes D)\circ \delta_{D}= (D\otimes \delta_{D})\circ \delta_{D}$, i.e., the coproduct $\delta_{D}$ is coassociative. If $D$ and
$E$ are counital comagmas (comonoids) in  ${\sf C}$, a morphism 
$f:D\rightarrow E$ in {\sf C} is called a  morphism of counital comagmas (comonoids) if $\varepsilon_{E}\circ f
=\varepsilon_{D}$ (i.e., the morphism $f$ is counitary),  and $(f\otimes f)\circ
\delta_{D} =\delta_{E}\circ f$ (i.e., the morphism $f$ is comultiplicative). .  

Moreover, if $D$, $E$ are counital comagmas (comonoids) in ${\sf C}$,
the object $D\otimes E$ is a counital comagma (comonoid) in ${\sf C}$, where
$\varepsilon_{D\otimes E}=\varepsilon_{D}\otimes \varepsilon_{E}$
and $\delta_{D\otimes E}=(D\otimes c_{D,E}\otimes E)\circ(
\delta_{D}\otimes  \delta_{E})$. If $D=(D, \varepsilon_{D}, \delta_{D})$ is a counital comagma so is $D^{cop}=(D, \varepsilon_{D}, c_{D,D}\co\delta_{D})$.

Let $f:D\rightarrow A$ and $g:D\rightarrow A$ be morphisms between a comagma $D$ and a magma $A$. We define the convolution product of $f$ and $g$ by $f\ast g=\mu_{A}\circ
(f\otimes g)\circ \delta_{D}$. If $A$ is unital and $D$ counital, we will say that $f$ is convolution invertible if there exists $f^{-1}:D\to A$ such that $f \ast f^{-1}=f^{-1}\ast f=\varepsilon_D\ot \eta_A$.

\begin{definition}
\label{nbimod}
{\rm A non-associative bimonoid in the category $\sf C$ is a unital magma $(H,\eta_H,\mu_H)$ and a comonoid $(H,\varepsilon_H,\delta_H)$ such that $\varepsilon_H$ and $\delta_H$ are morphisms of unital magmas (equivalently, $\eta_H$ and $\mu_H$ are morphisms of counital comagmas). Then the following identities hold:
\begin{equation}
\label{eta-eps}
\varepsilon_{H}\co \eta_{H}=id_{K},
\end{equation}
\begin{equation}
\label{mu-eps}
\varepsilon_{H}\co \mu_{H}=\varepsilon_{H}\ot \varepsilon_{H},
\end{equation}
\begin{equation}
\label{delta-eta}
\delta_{H}\co \eta_{H}=\eta_{H}\ot \eta_{H},
\end{equation}
\begin{equation}
\label{delta-mu}
\delta_{H}\co \mu_{H}=(\mu_{H}\ot \mu_{H})\co \delta_{H\ot H}.
\end{equation}
		
A non-associative bimonoid is called commutative if $\mu_{H}=\mu_{H}\circ c_{H,H}$, i.e., $H=H^{op}$ as unital magmas. Also, is called cocommutative if $\delta_{H}=c_{H,H}\co \delta_{H}$, i.e., $H=H^{cop}$ as comonoids.}
\end{definition}

\begin{definition}
\label{nbimod}
{\rm A non-coassociative bimonoid in the category $\sf C$ is a monoid $(D,\eta_D,\mu_D)$ and a counital comagma $(D,\varepsilon_D,\delta_D)$ such that $\eta_D$ and $\mu_D$ are morphisms of unital comagmas (equivalently, $\varepsilon_D$ and $\delta_D$ are morphisms of unital magmas). Then, as in the previous definition,  the  identities (\ref{eta-eps}), (\ref{mu-eps}), (\ref{delta-eta}) and (\ref{delta-mu}) hold.
		
A non-coassociative bimonoid is called  cocommutative if $\delta_{D}=c_{D,D}\co \delta_{D}$, i.e., $D=D^{cop}$ as counital comagmas. Also, is called commutative if $\mu_{D}=\mu_{D}\circ c_{D,D}$, i.e., $D=D^{op}$ as monoids. 
}
\end{definition}

Now we recall the notions of  Hopf quasigroup and Hopf coquasigroup in the category ${\sf C}$.

\begin{definition}
\label{Hopfqg} {\rm A Hopf quasigroup $H$ in
${\sf C}$ is a non-associative bimonoid such that there exists a morphism $\lambda_{H}:H\rightarrow H$ in ${\sf C}$ (called the  antipode of $H$) satisfying
\begin{equation}
\label{lH}
\mu_H\circ (\lambda_H\ot \mu_H)\circ (\delta_H\ot H)=\varepsilon_H\ot H=
\mu_H\circ (H\ot \mu_H)\circ (H\ot \lambda_H\ot H)\circ (\delta_H\ot H)
\end{equation}
and
\begin{equation}
\label{rH}
\mu_H\circ (\mu_H\ot H)\circ (H\ot \lambda_H\ot H)\circ (H\ot \delta_H)=
H\ot \varepsilon_H=
\mu_H\circ(\mu_H\ot \lambda_H)\circ (H\ot \delta_H). 
\end{equation}
		
Note that composing with $H\ot \eta_{H}$ in (\ref{lH}) we obtain that 
\begin{equation}
\label{1-conv}
\lambda_{H}\ast id_{H}=\varepsilon_H\ot \eta_H,  
\end{equation}
and composing with $ \eta_{H}\ot H$ in (\ref{rH}) we obtain 
\begin{equation}
\label{2-conv}
id_{H}\ast \lambda_H=\varepsilon_H\ot \eta_H. 
\end{equation}
		
Therefore, $\lambda_{H}$ is convolution invertible and $\lambda_{H}^{-1}=id_{H}$.

Definition \ref{Hopfqg}  is the monoidal version of the notion of Hopf quasigroup (also called non-associative Hopf algebra with the inverse property, or non-associative IP Hopf algebra) introduced in \cite{Majidesfera} (in this case {\sf C}=${\mathbb F}$-{\sf Vect}). Note that a  Hopf quasigroup $H$ is associative  if and only if  it is a Hopf algebra. 

If $H$ is a Hopf quasigroup in ${\sf C}$  we know that the antipode
$\lambda_{H}$ is unique, antimultiplicative, anticomultiplicative, i.e.,
\begin{equation} 
\label{antimu}
\lambda_{H}\co\mu_{H}=\mu_{H}\circ
(\lambda_{H}\ot\lambda_{H})\circ
c_{H,H},
\end{equation}
\begin{equation} 
\label{anticm}
\delta_{H}\co \lambda_{H}=
(\lambda_{H}\ot\lambda_{H})\circ
c_{H,H}\co \delta_{H}
\end{equation}
and leaves the unit and the counit invariable (see \cite{LP}):
\begin{equation} 
\label{inv}
\lambda_{H}\co\eta_{H}=\eta_{H},\;\;\;\varepsilon_{H}\co \lambda_{H}=\varepsilon_{H}. 
\end{equation}

Note that, by \cite[Proposition 4.3]{Majidesfera}, if $H$ is a commutative or cocommutative  Hopf quasigroup, we have that $\lambda_{H}^2=id_{H}$. Therefore, under (co)commutativity conditions, the  antipode of $H$ is an isomorphism.

A morphism between Hopf quasigroups $H$ and $A$ is a morphism $f:H\rightarrow A$ of unital magmas and comonoids. Then (see Lemma 1.4 of \cite{our1}) the equality
\begin{equation}
\label{antipode-morphism}
\lambda_A\co f=f\co \lambda_H
\end{equation}
holds.
		
}
\end{definition}

\begin{examples}
\label{exaHq}
{\rm  
A quasigroup is a set $Q$ together with a product such that for
any two elements $u, v \in Q$ the equations $u  x = v$, $x u = v$ and $u  v =
x$ have  unique solutions in $Q$. A quasigroup $L$ which contains an
element $e_{L}$ such that $ue_{L} = u = e_{L} u $ for every $u \in
L$ is called a loop. A loop $L$ is said to be a loop with the
inverse property  (for brevity an I.P. loop) if 
to every element $u\in L$, there corresponds an element $u^{-1}\in
L$ such that the equations $u^{-1}(uv)=v=(vu)u^{-1}$
hold for  every $v\in L$.
	
If $L$ is an I.P. loop, it is easy to show (see \cite{Bruck}) that
for all $u\in L$ the element $u^{-1}$ is unique and $
u^{-1}u=e_{L}=uu^{-1}.$
Moreover, the mapping $u\rightarrow u^{-1}$ is an anti-automorphism
of the I.P. loop $L$: $
(uv)^{-1}=v^{-1}u^{-1}.$ Then, $L$ is an I.P. loop iff $L$ is a cocommutative Hopf quasigroup in {\sf Set}. A concrete examples of these objects is the set of invertible elements of the octonions ${\mathbb O}$, the sphere $S^7$  (or in more general way the spheres  $S^{2^n-1}$) and the 16 Moufang loop ${\mathcal G}_{\mathbb O}$ associated to the octonions  (see \cite[Section 2]{Majidesfera}).

Let $R$ be a commutative ring and $L$ an I.P. loop. Then, by
\cite[Proposition 4.7]{Majidesfera}, we know that
$$RL=\bigoplus_{u\in L}Ru$$
is a cocommutative Hopf quasigroup in $R$-{\sf Mod} with  product defined by the
linear extension of the one defined in $L$ and
$\delta_{RL}(u)=u\ot u, \; \varepsilon_{RL}(u)=1_{R}, \;
\lambda_{RL}(u)=u^{-1}$ on the basis elements.

On the other hand, consider a commutative ring  $R$ with $\frac{1}{2}$ and $\frac{1}{3}$ in $R$. A Malcev algebra  $(M,[\;,\;])$ over $R$ is a free module over $R$ with a bilinear anticommutative
operation [\;,\;] on $M$ satisfying that:
\[[J(a, b, c), a] = J(a, b, [a, c]),\]
where $J(a, b, c) = [[a, b], c] - [[a, c], b] - [a, [b, c]]$ is the Jacobian in $a, b, c$ (see \cite{PIS}). By the construction given in \cite{PIS}, there exists a cocommutative Hopf quasigroup structure in $R$-{\sf Mod} associated to $M$. Indeed, consider the not necessarily associative algebra $U(M)$ defined as the quotient of $R\{M\}$, the free non-associative algebra on a basis of $M$, by the ideal $I(M)$ generated by the set
\[
\{ab-ba-[a,b], (a,x,y)+(x,a,y), (x,a,y)+(x,y,a) \;/\;   a, b\in M, x,y \in R\{M\} \},
\]
where $(x,y,z)=(xy)z-x(yz)$ is the usual additive associator. By  \cite[Proposition 4.1]{PIS} and \cite[Proposition 4.8]{Majidesfera}, the diagonal map $\delta_{U(M)}:U(M)\to U(M)\ot U(M)$ defined by $\delta_{U(M)} (x)=1\ot x +x\ot 1$  for all $x\in M,$ and the map $\varepsilon_{U(M)}:U(M)\to R$ defined by   $\varepsilon_{U(M)}(x)=0$   for all $x\in M$, both  extended to $U(M)$ as algebra morphisms; together with the map  $\lambda_{U(M)}:U(M)\to U(M)$, defined by $\lambda_{U(M)}(x)=-x$ for all $x\in M$ and extended to $U(M)$ as an antialgebra morphism, provide a cocommutative Hopf quasigroup structure on $U(M)$.

}
\end{examples}

\begin{definition}
\label{Hopfcqg} {\rm A Hopf coquasigroup $D$ in
${\sf C}$ is a non-coassociative bimonoid such that there exists a morphism $\lambda_{D}:D\rightarrow D$ in ${\sf C}$ (called the  antipode of $D$) satisfying
\begin{equation}
\label{clH}
(D\otimes \mu_{D})\circ (\delta_{D}\otimes \lambda_{D})\circ \delta_{D}=D\otimes \eta_{D}=(D\otimes \mu_{D})\circ  (D\ot \lambda_D\ot D)\circ (\delta_{D}\otimes D)\circ \delta_{D}
\end{equation}
and
\begin{equation}
\label{crH}
(\mu_{D}\otimes D)\circ (\lambda_{D}\otimes \delta_{D})\circ \delta_{D}=\eta_{D}\otimes D=(\mu_{D}\otimes D)\circ (D\ot \lambda_D\ot D)\circ (D\ot \delta_{D})\circ \delta_{D}.
\end{equation}
		
Note that composing with $\varepsilon_{D}\ot D$ in (\ref{clH}) we obtain (\ref{2-conv}) and composing with $ D\ot \varepsilon_{D}$ in (\ref{crH}) we obtain  (\ref{1-conv}). Then, as in the quasigroup case, 
$\lambda_{D}$ is convolution invertible and $\lambda_{D}^{-1}=id_{D}$.
}
\end{definition}

It is obvious that, a  Hopf coquasigroup $D$ is coassociative, i.e., $D$ is a comomonoid,  if and only if $D$ is a Hopf algebra. Moreover, as in the Hopf quasigroup case, if $D$ is a Hopf coquasigroup, the antipode
$\lambda_{D}$ is unique, antimultiplicative, anticomultiplicative and leaves the unit and the counit invariant. In this setting a morphism between two Hopf coquasigroups $D$ and $B$ is a morphism $g:D\rightarrow B$ of monoids and counital comagmas. Therefore, \ref{antipode-morphism} holds, i.e., $\lambda_B\co g=g\co \lambda_D.$

\begin{example}
Let ${\mathbb F}$ be a field. By \cite{Majidesfera} the algebraic variety ${\mathbb F}[S^{7}]$ is an example of Hopf coquasigroup. Also, there is a natural action of ${\mathbb Z}_{2}^{n}$ on ${\mathbb F}[S^{7}]$ which leads to a cross coproduct ${\mathbb F}[S^{7}]\rtimes {\mathbb Z}_{2}^{n}$ as the first example of noncommutative noncocommutative  Hopf coquasigroup (see \cite[Proposition 5.10, Example 5.11]{Majidesfera}). 

\end{example}

On the other hand, we can obtain examples of Hopf coquasigroups as duals of finite Hopf quasigroups. Following \cite{SAGA}, in the next definition we recall the notion of finite  object in  the category ${\sf C}$.

\begin{definition}
\label{prof}
{\rm An object  $P$ in ${\sf C}$ is finite if there exists $P^{\ast}$ in ${\sf C}$, called the dual object of $P$, such that $(P\otimes -, P^{\ast}\otimes -, \alpha_{P}, \beta_{P})$ is an adjoint pair.
	
If $f:P\rightarrow Q$ is a morphism between finite objects, we define the dual morphism of $f$ as $f^{\ast}: Q^{\ast}\rightarrow P^{\ast}$ where $f^{\ast}=(P^{\ast}\otimes (\beta_{Q}(K)\circ (f\otimes Q^{\ast})))\circ (\alpha_{P}(K)\otimes Q^{\ast})$.

}
\end{definition}

If $P$ and $Q$ are finite  objects,  
$P\otimes Q$ is a finite  object where  $(P\otimes Q)^{\ast}=Q^{\ast}\otimes P^{\ast}$ because, if $(P\otimes -, P^{\ast}\otimes -, \alpha_{P}, \beta_{P})$ 
and $(Q\otimes -, Q^{\ast}\otimes -, \alpha_{Q}, \beta_{Q})$ are adjoint 
pairs, then $$(P\otimes Q\otimes -, Q^{\ast}\otimes P^{\ast}\otimes -, 
\alpha_{P\otimes Q}, \beta_{P\otimes Q})$$ with 
$$ \alpha_{P\otimes Q}=(Q^{\ast}\otimes \alpha_{P}(K)\otimes Q\otimes -)
\circ (\alpha_{Q}(K)\otimes -)$$
and
$$\beta_{P\otimes Q}=
(\beta_{Q}(K)\otimes -)\circ 
(P\otimes \beta_{Q}(K)\otimes P^{\ast}\otimes -),$$
is an adjoint pair. Also, for morphisms we have that $(f\otimes g)^{\ast}=g^{\ast}\otimes f^{\ast}$. On the other hand, if $P$ 
is a finite object with adjoint pair $(P\otimes -, P^{\ast}\otimes -, \alpha_{P}, \beta_{P})$,   $P^{\ast}$ is finite object where   $P^{\ast \ast}= P$ because  $(P^{\ast}\otimes -, P\otimes -, \alpha_{P^{\ast}}, 
\beta_{P^{\ast}})$ with $\alpha_{P^{\ast}}=(c_{P^{\ast},P}\otimes -) 
\circ \alpha_{P}$ and $\beta_{P^{\ast}}=\beta_{P}\circ (
c_{P^{\ast},P}\otimes -)$ is an adjoint pair.

Then, by the properties of quoted in the previous paragraph, if $H$ is a finite Hopf quasigroup, it is easy to prove that the dual object $H^{\ast}$ is a Hopf coquasigroup with $\eta_{H^{\ast}}=\varepsilon_{H}^{\ast}$, $\mu_{H^{\ast}}=\delta_{H}^{\ast}$, $\varepsilon_{H^{\ast}}=\eta_{H}^{\ast}$, 
$\delta_{H^{\ast}}=\mu_{H}^{\ast}$ and $\lambda_{H^{\ast}}=\lambda_{H}^{\ast}$  as antipode.  Therefore, if $H$ is finite object in ${\sf C}$, $H$ is a Hopf quasigroup iff $H^{\ast}$ is a Hopf coquasigroup. Similarly, $H$ is a Hopf  coquasigroup iff $H^{\ast}$ is a Hopf quasigroup.

Finally, by \cite[Corollary 1]{LP}, we know that, if $H$ is a finite Hopf (co)quasigroup,  the antipode of $H$ is an isomorphism.

\section{Wreath (co)products for Hopf (co)quasigroups}

In this section we recall the main notions and results introduced and proved in \cite{RGR} about the wreath product of Hopf quasigroups. Following \cite{RGR}, this kind of products are the ones associated to $a$-comonoidal distributive laws.

\begin{definition}
\label{dl} {\rm Le $H$, $A$ be Hopf quasigroups. A morphism $\Psi:H\ot A\rightarrow A\ot H$ is said to be a distributive law of $H$ over $A$ if the following identities
\begin{equation}
\label{dl1}
\Psi\co (H\ot \mu_{A})\co (\lambda_{H}\ot \lambda_{A}\ot A)=(\mu_{A}\ot H)\co (A\ot \Psi)\co (\Psi\ot A)\co (\lambda_{H}\ot \lambda_{A}\ot A),
\end{equation}
\begin{equation}
\label{dl2}
\Psi\co (\mu_{H}\ot A)\co (H\ot \lambda_{H}\ot \lambda_{A})=(A\ot \mu_{H})\co ( \Psi\ot H)\co (H\ot \Psi)\co (H\ot \lambda_{H}\ot \lambda_{A}),
\end{equation}
\begin{equation}
\label{dl3}
\Psi\co (H\ot \eta_{A})=\eta_{A}\ot H,
\end{equation}
\begin{equation}
\label{dl4}
\Psi\co (\eta_{H}\ot A)=A\ot \eta_{H}, 
\end{equation}
hold.
}
\end{definition}

If   the antipodes of $H$ and $A$ are isomorphisms (for example, if $H$ and $A$ are finite objects), the identities (\ref{dl1}) and (\ref{dl2}) are equivalent to 
\begin{equation}
\label{dl1-1}
\Psi\co (H\ot \mu_{A})=(\mu_{A}\ot H)\co (A\ot \Psi)\co (\Psi\ot A),
\end{equation}
\begin{equation}
\label{dl2-1}
\Psi\co (\mu_{H}\ot A)=(A\ot \mu_{H})\co ( \Psi\ot H)\co (H\ot \Psi),
\end{equation}
respectively. Then, in this case, the conditions of the definition of distributive law for Hopf quasigroups are the ones that we can find in the classical definition of distributive law between monoids, i.e., $\Psi$ is compatible with the unit and the product of $A$ and $H$. 

\begin{definition}
\label{cdl} {\rm Le $H$, $A$ be Hopf quasigroups and let  $\Psi:H\ot A\rightarrow A\ot H$  be a distributive law of $H$ over $A$. The distributive law $\Psi$ is said to be comonoidal if it is a comonoid morphism, i.e., the following identities
\begin{equation}
\label{cdl1}
\delta_{A\ot H}\co \Psi=(\Psi \ot \Psi)\co \delta_{H\ot A},
\end{equation}
\begin{equation}
\label{cdl2}
(\varepsilon_{A}\ot \varepsilon_{H})\co\Psi=\varepsilon_{H}\ot \varepsilon_{A}, 
\end{equation}
hold.
}
\end{definition}

\begin{definition}
\label{cdl} {\rm Le $H$, $A$ be Hopf quasigroups and let  $\Psi:H\ot A\rightarrow A\ot H$  be a comonoidal distributive law of $H$ over $A$. We will say that $\Psi$ is  an $a$-comonoidal distributive law of $H$ over $A$  if the following identities
\begin{equation}
\label{adl1}
(A\ot \mu_{H})\co (\Psi\ot \mu_{H})\co (H\ot \Psi\ot H)\co (((\lambda_{H}\ot H)\co \delta_{H})\ot A\ot H)=\varepsilon_{H}\ot A\ot H,
\end{equation}
\begin{equation}
\label{adl2}
(A\ot \mu_{H})\co (\Psi\ot \mu_{H})\co (H\ot \Psi\ot H)\co (((H\ot \lambda_{H})\co \delta_{H})\ot A\ot H)=\varepsilon_{H}\ot A\ot H,
\end{equation}
\begin{equation}
\label{adl3}
(\mu_{A}\ot H)\co (\mu_{A}\ot \Psi)\co (A\ot \Psi\ot A)\co (A\ot H\ot (( \lambda_{A}\ot A)\co \delta_{A}))=A\ot H\ot \varepsilon_{A}, 
\end{equation}
\begin{equation}
\label{adl4}
(\mu_{A}\ot H)\co (\mu_{A}\ot \Psi)\co (A\ot \Psi\ot A)\co (A\ot H\ot (( A\ot \lambda_{A})\co \delta_{A}))=A\ot H\ot \varepsilon_{A},
\end{equation} 
hold.
}
\end{definition}

In \cite[Theorem 3.1]{RGR}  the author prove that,  for an $a$-comonoidal distributive law $\Psi$ of $H$ over $A$,   the wreath product $A\ot_{\Psi} H$ built on $A\ot H$ with  the  wreath product magma
\begin{equation}
\label{wpa1}
\mu_{A\ot_{\Psi}H}= (\mu_{A}\ot \mu_{H})\co (A\ot \Psi\ot H)
\end{equation} 
unit $\eta_{A\ot_{\Psi}H}=\eta_{A}\otimes \eta_{H}$, counit $\varepsilon_{A\ot_{\Psi}H}=\varepsilon_{A\otimes H}$, coproduct $\delta_{A\ot_{\Psi}H}=\delta_{A\otimes H}$ and antipode  
\begin{equation}
\label{64}
\lambda_{A\ot_{\Psi} H}=\Psi\co (\lambda_{A}\ot \lambda_{H})\co c_{H,A},
\end{equation}
is a Hopf quasigroup. 

\begin{example}
{\rm As was proved in \cite{RGR}, the Hopf quasigroups defined by the twisted double method in \cite{FT} are examples of wreath product Hopf quasigroups. Also, the smash products of Hopf quasigroups in the sense of \cite{BrzezJiao1} are examples of wreath product Hopf quasigroups as well as the twisted smash products defined in \cite{FT1}. 
}
\end{example}

\begin{example}
\label{1}
{\rm In \cite[Example 2.8]{RGR}  the author proved that the theory of double cross products of  Hopf quasigroups, introduced in \cite{our2}, provides an interesting family of $a$-comonoidal distributive laws. In the following lines we recall the  main details. 
	
Let  $H$ be a  Hopf quasigroup. The pair $(M,\varphi_{M})$ is said to be a  left $H$-quasimodule if $M$ is an object in ${\sf C}$ and $\varphi_{M}:H\ot M\rightarrow M$ is a morphism in ${\sf C}$ (called the action) satisfying
\begin{equation}
\label{uq}
\varphi_{M}\circ(\eta_{H}\ot M)=id_{M}
\end{equation}
and
\begin{equation}
\label{pq}
\varphi_{M}\co (H\ot\varphi_{M})\co  (((H\ot \lambda_{H})\co \delta_H)\ot M)=\varepsilon_H\ot M=\varphi_{M}\co (\lambda_H\ot \varphi_M)\co (\delta_H\ot M).
\end{equation}
		
Given two left ${H}$-quasimodules $(M,\varphi_{M})$, $(N,\varphi_{N})$ and a morphism  $f:M\rightarrow N$ in ${\sf C}$, we will say that $f$ is a morphism of left
$H$-quasimodules if 
\begin{equation}
\label{mor}
\varphi_{N}\co(H\ot f)=f\co\varphi_{M}.
\end{equation}

We will say that a unital magma $A$ is a left $H$-quasimodule magma  if it is a left $H$-quasimodule with action $\varphi_{A}: H\ot A\rightarrow A$ and the following equalities 
\begin{equation}
\label{etaq}
\varphi_{A}\circ(H\ot \eta_A)=\varepsilon_H\ot \eta_A,
\end{equation}
\begin{equation}
\label{muq}
\mu_A\co \varphi_{A\ot A}=\varphi_A\co (H\ot \mu_A),
\end{equation}
hold, i.e., $\eta_{A}$ and $\mu_{A}$ are quasimodule morphisms  where the action on $A\otimes A$ is defined by $\varphi_{A\ot A}=(\varphi_{A}\otimes \varphi_{A})\circ (H\otimes c_{H,A}\otimes A) \circ (\delta_{H}\otimes A\otimes A)$
		
A comonoid $A$ is a left $H$-quasimodule comonoid if it is a left $H$-quasimodule with action $\varphi_{A}$ and 
\begin{equation}
\label{eq}
\varepsilon_A\co \varphi_{A}=\varepsilon_H\ot \varepsilon_A,
\end{equation}
\begin{equation}
\label{dq}
\delta_A\co \varphi_{A}=\varphi_{A\ot A}\co (H\ot\delta_A),
\end{equation}
hold, i.e., $\varepsilon_{A}$ and $\delta_{A}$ are quasimodule morphisms.

Replacing (\ref{pq}) by the equality 
\begin{equation}
\label{pqmod}
\varphi_{M}\co(H\ot\varphi_{M})=\varphi_M\co (\mu_H\ot M),
\end{equation}
we have the definition of left $H$-module  and the ones of left $H$-module magma and comonoid. Note that the pair $(H, \mu_H)$ is not an $H$-module but it is an  $H$-quasimodule. Morphisms between left $H$-modules are defined as for $H$-quasimodules and we denote the category of left $H$-modules by $_{H}${\sf Mod}. Obviously we have similar definitions for the right side.
		
In \cite[Corollary 5.4]{our2} the authors  prove that, if $A$, $H$ are  Hopf quasigroups,  $(A, \varphi_A)$ is a left $H$-module comonoid, $(H, \phi_H)$  is a right $A$-module comonoid and 
$$\Psi=(\varphi_A\ot \phi_H)\co \delta_{H\ot A},$$ 
the following assertions are equivalent:
\begin{itemize}
\item [(i)] The double cross product $A\bowtie H$ built on the object $A\ot H$ with product
$$\mu_{A\bowtie H}=(\mu_A\ot \mu_H)\co (A\ot \Psi\ot H)$$
and tensor product unit, counit and coproduct, is a Hopf quasigroup with antipode
$$\lambda_{A\bowtie H}=\Psi\co (\lambda_H\ot \lambda_A)\co c_{A,H}.$$
			
\item [(ii)] The  equalities
\begin{equation}\label{d1}
\varphi_A\co (H\ot \eta_A)=\varepsilon_H\ot \eta_A,
\end{equation}
\begin{equation}\label{d2}
\phi_H\co (\eta_H\ot A)=\eta_H\ot \varepsilon_A,
\end{equation}
\begin{equation}\label{d3}
(\phi_H\ot \varphi_A)\co \delta_{H\ot A}=c_{A,H}\co \Psi,
\end{equation}
\begin{equation}\label{d4}
\varphi_A\co (H\ot \mu_A)\co (\lambda_H\ot \lambda_A\ot A)=\mu_A\co(A\ot \varphi_A)\co ((\Psi\co (\lambda_H\ot \lambda_A))\ot A),
\end{equation}
\begin{equation}\label{d5}
\mu_H\co(\phi_H\ot \mu_H)\co (\lambda_H\ot \Psi \ot H)\co  (\delta_H\ot A\ot H)=\varepsilon_H\ot \varepsilon_A\ot H,
\end{equation}
\begin{equation}\label{d6}
\mu_H\co(\phi_H\ot \mu_H)\co (H\ot \Psi\ot H)\co  (((H\ot \lambda_H)\co\delta_H)\ot A\ot H)=\varepsilon_H\ot \varepsilon_A\ot H,
\end{equation}
\begin{equation}\label{d7}
\phi_H\co (\mu_{H}\ot A)\co (H\ot\lambda_H\ot \lambda_A)=\mu_H\co(\phi_{H}\ot H)\co (H\ot (\Psi\co (\lambda_H\ot \lambda_A))),
\end{equation}
\begin{equation}\label{d8}
\mu_A\co(\mu_{A}\ot \varphi_A)\co (A\ot \Psi\ot \lambda_{A})\co (A\ot H\ot \delta_{A})=A\ot \varepsilon_H\ot \varepsilon_A,
\end{equation}
\begin{equation}\label{d9}
\mu_{A}\co (\mu_{A}\ot \varphi_{A})\co (A\ot \Psi\ot A)\co (A\ot H\ot ((\lambda_{A}\ot A)\co \delta_{A}))=A\ot \varepsilon_H\ot \varepsilon_A,
\end{equation}
hold.
\end{itemize}
		
Under the conditions (\ref{d1})-(\ref{d9}), we can prove that $\Psi$ is an example of $a$-comonoidal distributive law of $H$ over $A$ (see \cite[Example 2.8]{RGR}) and then $A\bowtie H$ an example of wreath product associated to $\Psi$. 

Using the terminology that can be found in the literature on double cross products of Hopf algebras, we will say that, if $A$ and $H$ are in the conditions of this example,  $(A,H)$ is a matched pair of Hopf quasigroups.

Following, \cite[Definition 4.12]{Majidesfera}, we will say that a Hopf quasigroup $A$ is a left $H$-module Hopf quasigroup if it is a a left $H$-module monoid and comonoid. By, \cite[Proposition 4.14]{Majidesfera}, we know that if we denote the action of $H$ over $A$ by $\varphi_{A}$ and we assume that $H$ is cocommutative, there is a left cross product Hopf quasigroup $A\rtimes H$ built on $A\otimes H$ with tensor coproduct and unit and 
$$\mu_{A\rtimes H}=(\mu_{A}\otimes \mu_{H})\circ (A\otimes \Psi\otimes H), \;\;\; \mu_{A\rtimes H}=\Psi\co (\lambda_H\ot \lambda_A)\co c_{A,H}$$
where $\Psi=(\varphi_{A}\otimes H)\circ (H\otimes c_{H,A})\circ (\delta_{H}\otimes A)$. Then, the Hopf quasigroup $A\rtimes H$ is an example of double cross product of Hopf quasigroups with $\phi_{H}=H\otimes \varepsilon_{A}$.

}
\end{example}

The previous definitions and results can be extended to the wreath coproduct setting of Hopf coquasigroups by dualization as follows. 

\begin{definition}
\label{dlc} {\rm Le $D$, $B$ be Hopf coquasigroups. A morphism $\Omega:D\ot B\rightarrow B\ot D$ is said to be a codistributive law of $D$ over $B$ if the following identities
\begin{equation}
\label{dl1c}
(B\otimes \lambda_{B}\otimes \lambda_{D})\circ (\delta_{B}\otimes D)\circ \Omega=(B\otimes \lambda_{B}\otimes \lambda_{D})\circ (B\otimes \Omega)\circ (\Omega\otimes B)\circ (D\otimes \delta_{B}),
\end{equation}
\begin{equation}
\label{dl2c}
(\lambda_{B}\otimes \lambda_{D}\otimes D)\circ (B\otimes \delta_{D})\circ \Omega=(\lambda_{B}\otimes \lambda_{D}\otimes D)\circ (\Omega\otimes D)\circ (D\otimes \Omega)\circ (\delta_{D}\otimes B),
\end{equation}
\begin{equation}
\label{dl3c}
(\varepsilon_{B}\otimes D)\circ \Omega=D\otimes \varepsilon_{B},
\end{equation}
\begin{equation}
\label{dl4c}
(B\otimes \varepsilon_{D})\circ \Omega=\varepsilon_{D}\otimes B, 
\end{equation}
hold.
	}
\end{definition}

If  the antipodes of $D$ and $B$ are isomorphisms (for example, if $D$ and $B$ are finite objects), the identities (\ref{dl1c}) and (\ref{dl2c}) are equivalent to 
\begin{equation}
\label{dl1-1c}
(\delta_{B}\otimes D)\circ \Omega=(B\otimes \Omega)\circ (\Omega\otimes B)\circ (D\otimes \delta_{B}),
\end{equation}
\begin{equation}
\label{dl2-1c}
(B\otimes \delta_{D})\circ \Omega=(\lambda_{B}\otimes \lambda_{D}\otimes D)\circ (\Omega\otimes D)\circ (D\otimes \Omega)\circ (\delta_{D}\otimes B),
\end{equation}
respectively. Then, in this case, the conditions of the definition of codistributive law for Hopf quasigroups are the ones that we can find for  codistributive laws between comonoids, i.e., $\Omega$ is compatible with the counit and the coproduct of $D$ and $B$. 

\begin{definition}
\label{cdlc} {\rm Le $D$, $B$ be Hopf coquasigroups and let  $\Omega:D\ot B\rightarrow B\ot D$  be a codistributive law of $D$ over $B$. The codistributive law $\Omega$ is said to be monoidal if it is a monoid morphism, i.e., the following identities
\begin{equation}
\label{cdl1c}
\Omega\circ \mu_{D\ot B}=\mu_{B\otimes D}\circ (\Omega \ot \Omega),
\end{equation}
\begin{equation}
\label{cdl2c}
\Omega\circ (\eta_{D}\ot \eta_{B}) =\eta_{B}\ot \eta_{D}, 
\end{equation}
hold.
}
\end{definition}

\begin{definition}
\label{cdlc} {\rm Le $D$, $B$ be Hopf coquasigroups and let  $\Omega:D\ot B\rightarrow B\ot D$   be a monoidal codistributive law of $D$ over $B$. We will say that $\Omega$ is  an $a$-monoidal codistributive law of $D$ over $B$  if the following identities
\begin{equation}
\label{adl1c}
(D\otimes B\otimes (\mu_{D}\circ (D\otimes \lambda_{D}))\circ (D\otimes \Omega\otimes D)\circ (\delta_{D}\otimes \Omega)\circ (\delta_{D}\otimes B)=D\otimes B\otimes \eta_{D},
\end{equation}
\begin{equation}
\label{adl2c}
(D\otimes B\otimes (\mu_{D}\circ (\lambda_{D}\otimes D))\circ (D\otimes \Omega\otimes D)\circ (\delta_{D}\otimes \Omega)\circ (\delta_{D}\otimes B)=D\otimes B\otimes \eta_{D},
\end{equation}
\begin{equation}
\label{adl3c}
((\mu_{B}\circ (B\otimes \lambda_{B}))\otimes D\otimes B)\circ (B\otimes \Omega\otimes B)\circ (\Omega\otimes \delta_{B})\circ (D\otimes \delta_{B})=\eta_{B}\otimes D\otimes B, 
\end{equation}
\begin{equation}
\label{adl4c}
((\mu_{B}\circ (\lambda_{B}\otimes B))\otimes D\otimes B)\circ (B\otimes \Omega\otimes B)\circ (\Omega\otimes \delta_{B})\circ (D\otimes \delta_{B})=\eta_{B}\otimes D\otimes B,
\end{equation} 
hold.
}
\end{definition}

Note that, if $D$ and $B$ are Hopf algebras and $\Omega:D\ot B\rightarrow B\ot D$ is a codistributive law between the comonoids $D$ and $B$, the equalities (\ref{adl1c}), (\ref{adl2c}), (\ref{adl3c}) and (\ref{adl4c}) always hold. On the other hand, it is obvious that, if $H$, $A$ are finite Hopf quasigroups, $\Psi$ is an $a$-comonoidal distributive law of $H$ over $A$ iff $\Psi^{\ast}$ is an $a$-monoidal codistributive law of $H^{\ast}$ over $A^{\ast}$. 

If we dualize \cite[Theorem 3.1]{RGR} we have the following: For an $a$-monoidal codistributive law $\Omega:D\ot B\rightarrow B\ot D$ of $D$ over $B$, the wreath coproduct $D\ot_{\Omega} B$ built on $D\ot B$ with the wreath coproduct comagma
\begin{equation}
\label{wpa1}
\delta_{D\ot_{\Omega} B}=(D\otimes \Omega\otimes B)\circ (\delta_{D}\otimes \delta_{B})
\end{equation} 
 counit $\varepsilon_{D\ot_{\Omega} B}=\varepsilon_{D}\otimes \varepsilon_{B}$,  unit $\eta_{D\ot_{\Omega} B}=\eta_{D\otimes B}$, product $\mu_{D\otimes_{\Omega} B}=\mu_{D\otimes B}$ and antipode  
\begin{equation}
\label{64}
\lambda_{D\otimes_{\Omega} B}=c_{B,D}\circ (\lambda_{B}\ot \lambda_{D})\circ \Omega,
\end{equation}
is a Hopf coquasigroup.

\begin{example}
\label{1c}
{\rm The dual of Example \ref{1} provides a family of $a$-monoidal codistributive laws. In this case we work with the theory of double cross coproducts of  Hopf coquasigroups. 
		
Let  $D$ be a  Hopf coquasigroup. The pair $(P,\rho_{P})$ is said to be a right $D$-quasicomodule if $P$ is an object in ${\sf C}$ and $\rho_{P}:P\rightarrow P\otimes D$ is a morphism in ${\sf C}$ (called the coaction) satisfying
\begin{equation}
\label{uqc}
(P\otimes \varepsilon_{D})\co \rho_{P}=id_{P}
\end{equation}
and
\begin{equation}
\label{pqc}
(P\otimes (\mu_{D}\circ (\lambda_{D}\otimes D)))\circ (\rho_{P}\otimes D)\circ \rho_{P}=P\otimes \eta_{D}=(P\otimes \mu_{D})\circ (\rho_{P}\otimes \lambda_{D})\circ \rho_{P} .
\end{equation}
		
Given two right $D$-quasicomodules $(P,\rho_{P})$, $(Q,\rho_{Q})$ and one morphism  $g:P\rightarrow Q$ in ${\sf C}$, we will say that $f$ is a morphism of right
$H$-quasicomodules if 
\begin{equation}
\label{morc}
(g\otimes D)\circ \rho_P=\rho_{Q}\circ g.
\end{equation}

We will say that a counital comagma $B$ is a right $D$-quasicomodule comagma  if it is a right $D$-quasicomodule with coaction $\rho_{B}: B\rightarrow B\otimes D$ and the following equalities 
\begin{equation}
\label{etaqc}
(\varepsilon_{B}\otimes D)\circ \rho_{B}=\varepsilon_B\ot \eta_D,
\end{equation}
\begin{equation}
\label{muqc}
\rho_{B\otimes B}\circ \delta_{B}=(\delta_{B}\otimes D)\circ \rho_{B},
\end{equation}
hold, i.e., $\varepsilon_{B}$ and $\delta_{B}$ are morphisms of right quasicomodules where the coaction on $B\otimes B$  is defined by  $\rho_{B\otimes B}=(B\otimes B\otimes \mu_{D})\circ (B\otimes c_{B,D}\otimes D)\circ (\rho_{B}\otimes \rho_{B}).$
		
A monoid $B$ is a right $D$-quasicomodule monoid if it is a right $D$-quasicomodule with coaction $\rho_{B}$ and 
\begin{equation}
\label{eqc}
\rho_{B}\circ \eta_{B}=\eta_{B}\otimes \eta_{D},
\end{equation}
\begin{equation}
\label{dqc}
\rho_{B}\circ \mu_{B}=(\mu_{B}\otimes D)\circ \rho_{B\otimes B},
\end{equation}
hold, i.e., $\eta_{B}$ and $\mu_{B}$ are right quasicomodule morphisms.

Replacing (\ref{pqc}) by the equality 
\begin{equation}
\label{pqmodc}
(\rho_{P}\otimes D)\circ \rho_{P}=(P\otimes \delta_{D})\circ \rho_{P},
\end{equation}
we have the definition of right $D$-comodule  and the ones of right $D$-comodule comagma and monoid. Note that the pair $(D, \delta_D)$ is not an $D$-comodule but it is an  $D$-quasicomodule. Morphisms between right $D$-comodules are defined as for $D$-quasicomodules and we denote the category of right $D$-comodules by {\sf Mod}$^{D}$. Obviously we have similar definitions for the left side.
		
Then by the dual proof of \cite[Corollary 5.4]{our2} we can prove the following. Let $B$, $D$ be  Hopf coquasigroups. Let $(B, \rho_B)$ be a right $D$-comodule monoid and  let $(D, r_D)$  be a  left $B$-comodule monoid and write
$$\Omega=\mu_{B\otimes D}\circ (r_{D}\otimes \rho_{B}).$$ 
Then, the following assertions are equivalent:
		
\begin{itemize}
\item [(i)] The double cross coproduct $D\infty B$ built on the object $D\otimes B$ with coproduct
$$\delta_{D\infty B}=(D\otimes \Omega\otimes B)\circ (\delta_{D}\otimes \delta_{B})$$
and tensor product counit, unit and product, is a Hopf coquasigroup with antipode
$$\lambda_{D\infty B}=c_{B,D}\circ (\lambda_{B}\otimes \lambda_{D})\circ \Omega.$$
			
\item [(ii)] The  equalities
\begin{equation}\label{d1c}
(\varepsilon_{B}\otimes D)\circ \rho_{B}=\varepsilon_{B}\otimes \eta_{D},
\end{equation}
\begin{equation}\label{d2c}
(B\otimes \varepsilon_{D})\circ r_{D}=\eta_{B}\otimes \varepsilon_{D},
\end{equation}
\begin{equation}\label{d3c}
\mu_{B\otimes D}\circ (\rho_{B}\otimes r_{D})=\Omega\circ c_{B,D},
\end{equation}
\begin{equation}\label{d4c}
(B\otimes \lambda_{B}\otimes \lambda_{D})\circ (\delta_{B}\otimes D)\circ \rho_{B}=(B\otimes ((\lambda_{B}\otimes \lambda_{D})\circ \Omega)\circ (\rho_{B}\otimes B)\circ \delta_{B},
\end{equation}
\begin{equation}\label{d5c}
(D\otimes B\otimes \mu_{D})\circ (D\otimes \Omega\otimes \lambda_{D})\circ (\delta_{D}\otimes r_{D})\circ \delta_{D}=D\otimes \eta_{B}\otimes \eta_{D},
\end{equation}
\begin{equation}\label{d6c}
(D\otimes B\otimes (\mu_{D}\circ (\lambda_{D}\otimes D)))\circ (D\otimes \Omega\otimes D)\circ (\delta_{D}\otimes r_{D})\circ \delta_{D}=D\otimes \eta_{B}\otimes \eta_{D},
\end{equation}
\begin{equation}\label{d7c}
(\lambda_{B}\otimes \lambda_{D}\otimes D)\circ (B\otimes \delta_{D})\circ r_{D}=(((\lambda_{B}\otimes \lambda_{D})\circ \Omega)\otimes D)\circ (D\otimes r_{D})\circ \delta_{D},
\end{equation}
\begin{equation}\label{d8c}
(\mu_{B}\otimes D\otimes B)\circ (\lambda_{B}\otimes \Omega\otimes B)\circ (\rho_{B}\otimes \delta_{B})\circ \delta_{B}=\eta_{B}\otimes \eta_{D}\otimes B,
\end{equation}
\begin{equation}\label{d9c}
((\mu_{B}\circ (B\otimes \lambda_{B}))\otimes D\otimes B)\circ (B\otimes \Omega\otimes B)\circ (\rho_{B}\otimes \delta_{B})\circ \delta_{B}=\eta_{B}\otimes \eta_{D}\otimes B,
\end{equation}
hold.
\end{itemize}
		
Under the conditions (\ref{d1c})-(\ref{d9c}), we can prove that $\Omega$ is an example of $a$-monoidal codistributive law of $D$ over $B$ and then $D\infty B$ an example of wreath coproduct associated to $\Omega$.

As in the Hopf quasigroup setting, if $B$ and $D$ are in the conditions of this example, we will say that  $(B,D)$ is a comatched pair of Hopf coquasigroups.

Recall that  cross coproducts Hopf coquasigroups (see \cite[Definition 5.12, Proposition 5.14]{Majidesfera}) are  examples of  double cross coproducts of  Hopf coquasigroups because is the dual of  cross products Hopf quasigroups (see the last paragraph of Example \ref{1}).
		
}
\end{example}

\section{Factorizations}

In the classical theory of Hopf algebras over a field ${\mathbb F}$ it is said that a Hopf algebra $X$ factorises as $X=AH$ if there exists Hopf subalgebras $A$ and $H$, with inclusions morphisms $i_{A}:A\rightarrow X$, $i_{H}:H\rightarrow X$, such that  the morphism $\omega_{X}=\mu_{X}\circ (i_{A}\otimes i_{H}):A\otimes H\rightarrow X$ is an isomorphism. Dually, if $X$ is finite dimensional, $\omega_{X}$ is an isomorphism iff the dual morphism $\omega_{X}^{\ast}=(i_{H}^{\ast}\otimes i_{A}^{\ast})\circ \delta_{B^{\ast}}$ is an  isomorphism. By \cite[Theorem 7.2.3]{MAJDCP}, $X=AH$ iff $X=A\bowtie H$, i.e., $X$ is obtained as a double cross product of Hopf algebras (or $(A, H)$ is a matched pair of Hopf algebras), equivalently, $X^{\ast}=H^{\ast}\infty A^{\ast}$ ($(A^{\ast}, H^{\ast})$ is a comatched pair of Hopf algebras). 
	
The main target of this section is to prove that for Hopf quasigroups we have similar results. In other words, the exact factorization problem for Hopf quasigroups is equivalent to the existence of a double cross product of Hopf quasigroups.

\begin{definition}
\label{factq}
{\rm Let $X$ be a Hopf quasigroup in {\sf C}. Let $H$, $A$ be Hopf subquasigroups of $X$ with inclusion morphisms $i_{H}:H\rightarrow X$, $i_{A}:A\rightarrow X$ respectively. Let  $\omega_{X}$ and $\theta_{X}$ be the morphisms defined by 
$$\omega_{X}=\mu_{X}\circ (i_{A}\otimes i_{H}):A\otimes H\rightarrow X, \;\;\;\theta_{X}=\mu_{X}\circ (i_{H}\otimes i_{A}):H\otimes A\rightarrow X.$$

We will say that  $X$ factorizes as $X=AH$ if $\omega_{X}$ is an isomorphism and the following identities
\begin{equation}\label{factq1}
\mu_{X}\circ (\omega_{X}\otimes X)=\mu_{X}\circ (i_{A}\otimes (\mu_{X}\circ (i_{H}\otimes X))),
\end{equation}
\begin{equation}\label{factq2}
\mu_{X}\circ (X\otimes\omega_{X})=\mu_{X}\circ ((\mu_{X}\circ (X\otimes i_{A}))\otimes i_{H}),
\end{equation}
\begin{equation}\label{factq3}
\mu_{X}\circ ((\theta_{X}\circ (\lambda_{H}\otimes \lambda_{A}))\otimes X)=\mu_{X}\circ ((i_{H}\circ \lambda_{H})\otimes (\mu_{X}\circ ((i_{A}\circ \lambda_{A})\otimes X))),
\end{equation}
\begin{equation}\label{factq4}
\mu_{X}\circ (X\otimes (\theta_{X}\circ (\lambda_{H}\otimes \lambda_{A})))=\mu_{X}\circ ((\mu_{X}\circ (X\otimes (i_{H}\circ \lambda_{H} ))\otimes (i_{A}\circ\lambda_{A}))
\end{equation}
hold.

Note that, if $H$ and $A$ are finite we can remove the antipodes in (\ref{factq3}) and (\ref{factq4}). Then these identities become  in 
\begin{equation}\label{factq31}
\mu_{X}\circ (\theta_{X}\otimes X)=\mu_{X}\circ (i_{H}\otimes (\mu_{X}\circ (i_{A}\otimes X)),
\end{equation}
\begin{equation}\label{factq41}
\mu_{X}\circ (X\otimes \theta_{X})=\mu_{X}\circ ((\mu_{X}\circ (X\otimes i_{H}))\otimes i_{A}).
\end{equation}

Finally, note that $\omega_{X}$ and $\theta_{X}$ are comonoid morphisms because $i_{A}$, $i_{H}$ are comonoid morphisms and (\ref{mu-eps}) and (\ref{delta-mu}) holds for the Hopf quasigroup $X$. Obviously, if $\omega_{X}$ is a comonoid isomorphism, its dual is a monoid isomorphism.

}
\end{definition}

\begin{example}
{\rm 
\label{ex1fact}  Suppose that $(A, H)$ is a matched pair of Hopf quasigroups.
By Example \ref{1}, the double cross product $A\bowtie H$ is a Hopf quasigroup.  The morphisms $i_{A}=A\otimes \eta_{H}:A\rightarrow A\bowtie H$ and $i_{H}=\eta_{A}\otimes H:H\rightarrow A\bowtie H$ are morphisms of Hopf quasigroups because (\ref{dl3}), (\ref{dl4})
holds and (\ref{eta-eps}) and (\ref{delta-eta}) also hold for $H$ and $A$. By the properties of the units and (\ref{dl4}), we obtain that 
$$\omega_{A\bowtie H}=id_{A\otimes H}, \;\;\; \theta_{A\bowtie H}=\Psi.$$

Then, by (\ref{dl4}),   
$$\mu_{A\bowtie H}\circ (i_{A}\otimes (\mu_{A\bowtie H}\circ (i_{H}\otimes A\bowtie H)))=\mu_{A\bowtie H}=\mu_{A\bowtie H}\circ (\omega_{A\bowtie H}\otimes A\bowtie H)$$
and, (\ref{factq1}) holds. Similarly, by (\ref{dl3}), we obtain (\ref{factq2}) because 
$$\mu_{A\bowtie H}\circ ((\mu_{A\bowtie H}\circ (A\bowtie H\otimes i_{A}))\otimes i_{H})=
\mu_{A\bowtie H}=\mu_{A\bowtie H}\circ (A\bowtie H\otimes\omega_{A\bowtie H}).$$

On the other hand, 

\begin{itemize}
\item[ ]$\hspace{0.38cm} \mu_{A\bowtie H}\circ ((\theta_{A\bowtie H}\circ (\lambda_{H}\otimes \lambda_{A}))\otimes A\bowtie H)$
	
\item[ ]$=(\mu_{A}\otimes \mu_{H})\circ (A\otimes \Psi\otimes H)\circ ((\Psi\circ (\lambda_{H}\otimes \lambda_{A}))\otimes A\otimes H) $ {\scriptsize ({\blue by the unit properties})} 
	
\item[ ]$= (A\otimes \mu_{H})\circ ((\Psi\circ (H\otimes \mu_{A})\circ (\lambda_{H}\otimes \lambda_{A}\otimes A))\otimes H)$  {\scriptsize  ({\blue by (\ref{dl1})})}
	
\item[ ]$=\mu_{A\bowtie H}\circ ((i_{H}\circ \lambda_{H})\otimes (\mu_{A\bowtie H}\circ ((i_{A}\circ \lambda_{A})\otimes A\bowtie H)))$ {\scriptsize  ({\blue (\ref{dl4})}),}
	
\end{itemize}

and (\ref{factq3}) holds. Similarly, 

\begin{itemize}
\item[ ]$\hspace{0.38cm}\mu_{A\bowtie H}\circ (A\bowtie H\otimes (\theta_{A\bowtie H}\circ (\lambda_{H}\otimes \lambda_{A})))$
	
\item[ ]$= (\mu_{A}\otimes \mu_{H})\circ (A\otimes \Psi\otimes H)\circ (A\otimes H\otimes (\Psi\circ (\lambda_{H}\otimes \lambda_{A})))$ {\scriptsize ({\blue by the unit properties})} 
	
\item[ ]$=(\mu_{A}\otimes H)\circ (A\otimes (\Psi\circ (\mu_{H}\otimes A)\circ (H\otimes \lambda_{H}\otimes \lambda_{A})))$  {\scriptsize  ({\blue by (\ref{dl2})})}
	
\item[ ]$=\mu_{A\bowtie H}\circ ((\mu_{A\bowtie H}\circ (A\bowtie H\otimes (i_{H}\circ \lambda_{H})))\otimes (i_{A}\circ \lambda_{A})) $ {\scriptsize  ({\blue by (\ref{dl3})}),}
	
\end{itemize}
and, as a consequence, (\ref{factq4}) also holds.

Therefore, any double cross product of Hopf quasigroups induces an example of factorization. 

}
\end{example}

\begin{theorem}
\label{prinprev} Let $H$, $A$ be Hopf subquasigroups of  a Hopf quasigroup $X$ with inclusion morphisms $i_{H}:H\rightarrow X$, $i_{A}:A\rightarrow X$ respectively. If  $X$ factorises as $X=AH$, the morphism 
$$\Psi=\omega_{X}^{-1}\circ \theta_{X}:H\otimes A\rightarrow A\otimes H$$
is a comonoidal distributive law of $H$ over $A$. Moreover, if the antipodes of $H$ and $A$ are isomorphisms $\Psi$ is an $a$-comonoidal distributive law of $H$ over $A$.
\end{theorem}

\begin{proof}
The condition (\ref{dl1})  of Definition \ref{dl} holds because: 
\begin{itemize}
\item[ ]$\hspace{0.38cm}  \omega_{X}\circ (\mu_{A}\ot H)\co (A\ot \Psi)\co (\Psi\ot A)\co (\lambda_{H}\ot \lambda_{A}\ot A) $
	
\item[ ]$= \mu_{X}\circ ((\mu_{X}\circ (i_{A}\otimes i_{A}))\otimes i_{H})\circ (A\otimes (\omega_{X}^{-1}\circ \theta_{X})) \circ ((\omega_{X}^{-1}\circ \theta_{X}\circ (\lambda_{H}\otimes \lambda_{A}))\otimes A)$ {\scriptsize ({\blue by definition}} 
\item[ ]$\hspace{0.38cm}$ {\scriptsize {\blue  of $\omega_{X}$ and the condition of monoid morphism for $i_{A}$)} }
	
\item[ ]$= \mu_{X}\circ (i_A\otimes (\omega_{X}\circ \omega_{X}^{-1}\circ \theta_{X})) \circ ((\omega_{X}^{-1}\circ \theta_{X}\circ (\lambda_{H}\otimes \lambda_{A}))\otimes A) $  {\scriptsize  ({\blue by (\ref{factq2})})}
	
\item[ ]$=\mu_{X}\circ (i_{A}\otimes (\mu_{X}\circ (i_{H}\otimes i_{A}))) \circ  ((\omega_{X}^{-1}\circ \theta_{X}\circ (\lambda_{H}\otimes \lambda_{A}))\otimes A)$ {\scriptsize  ({\blue by definition of $\theta_{X}$})}

\item[ ]$= \mu_{X} \circ  ((\omega_{X}\circ \omega_{X}^{-1}\circ \theta_{X}\circ (\lambda_{H}\otimes \lambda_{A}))\otimes i_A) $ {\scriptsize ({\blue by (\ref{factq1})})} 

\item[ ]$=\mu_{X}\circ ((i_{H}\circ \lambda_{H})\otimes (\mu_{X}\circ ((i_{A}\circ \lambda_{A})\otimes i_{A}))) $  {\scriptsize  ({\blue by (\ref{factq3})})}

\item[ ]$=\theta_{X}\circ (\lambda_{H}\otimes (\mu_{A}\circ (\lambda_{A}\otimes A))) $ {\scriptsize  ({\blue the condition of monoid morphism for $i_{A}$})}

\item[ ]$= \omega_{X}\circ \Psi\co (H\ot \mu_{A})\co (\lambda_{H}\ot \lambda_{A}\ot A)$ {\scriptsize  ({\blue by the condition of isomorphism for $\omega_{X}$}).}
	
\end{itemize}

By a similar proof, using that $i_{H}$ is a monoid morphism,  (\ref{factq1}) instead  (\ref{factq2}) ,   (\ref{factq2}) instead  (\ref{factq1})  and  (\ref{factq4}) instead  (\ref{factq3}) , we obtain that $\Psi$ satisfies (\ref{dl2}) of Definition \ref{dl}.  Also, by the unit properties and the condition of monoid morphism for $i_{A}$ we have that 
$$\Psi\circ (H\otimes \eta_{A})=\omega_{X}^{-1}\circ i_{H}=\omega_{X}^{-1}\circ \omega_{X}\circ (\eta_{A}\otimes H) =\eta_{A}\otimes H$$
and this implies that  (\ref{dl3})  of Definition \ref{dl} holds. In the same way, by the unit properties and the condition of monoid morphism for $i_{H}$, we prove that   (\ref{dl4})  of Definition \ref{dl} holds. Therefore,  $\Psi$  is distributive law of $H$ over $A$.

On the other hand, $\Psi$ is comonoidal because it is a composition of comonoid morphisms.

Finally, if the antipodes of $H$ and $A$ are isomorphisms, $\Psi$ is $a$-comonoidal distributiva law beacuse, as in the previous lines, if we compose with $\omega_{X}$, we have 

\begin{itemize}
\item[ ]$\hspace{0.38cm}  \omega_{X}\circ (A\ot \mu_{H})\co (\Psi\ot \mu_{H})\co (H\ot \Psi\ot H)\co (((\lambda_{H}\ot H)\co \delta_{H})\ot A\ot H) $
	
\item[ ]$= \mu_{X}\circ (\theta_{X}\otimes (\mu_{X}\circ (i_{H}\otimes i_{H})))\circ (\lambda_{H}\otimes (\omega_{X}^{-1}\circ \theta_{X})\otimes H)\circ (\delta_{H}\otimes A\otimes H)$ {\scriptsize ({\blue by the condition of } }
\item[ ]$\hspace{0.38cm}$ {\scriptsize {\blue monoid morphism for $i_{H}$ and  (\ref{factq1}))} }
	
\item[ ]$=\mu_{X}\circ ((i_{H}\circ \lambda_{H})\otimes (\mu_{X}\circ (i_{A}\otimes (\mu_{X}\circ (i_{H}\otimes i_{H}))))) \circ (H\otimes (\omega_{X}^{-1}\circ \theta_{X})\otimes H)\circ (\delta_{H}\otimes A\otimes H) $  
\item[ ]$\hspace{0.38cm}${\scriptsize  ({\blue by (\ref{factq31})})}
	
\item[ ]$=\mu_{X}\circ ((i_{H}\circ \lambda_{H})\otimes (\mu_{X}\circ (\theta_{X}\otimes  i_{H})))\circ (\delta_{H}\otimes A\otimes H)  $ {\scriptsize  ({\blue by (\ref{factq1})})}
	
\item[ ]$= \mu_{X}\circ ((\lambda_{X}\circ i_{H})\otimes (\mu_{X}\circ (i_{H}\otimes X)))\circ (\delta_{H}\otimes \omega_{X})  $ {\scriptsize ({\blue by (\ref{factq31}) and (\ref{antipode-morphism}) for $i_{H}$})} 
	
\item[ ]$= \mu_X\circ (\lambda_X\ot \mu_X)\circ ((\delta_X\circ i_{H})\ot \omega_{X})$  {\scriptsize  ({\blue by the condition of comonoid morphism for $i_{H}$})}
	
\item[ ]$=(\varepsilon_{X}\circ i_{H}) \otimes \omega_{X}$ {\scriptsize  ({\blue by (\ref{lH}) for $X$})}
	
\item[ ]$= \varepsilon_{H}\otimes \omega_{X} $ {\scriptsize  ({\blue by the condition of comonoid morphism for $i_{H}$}).}
	
\end{itemize}

and then, (\ref{adl1}) of Definition \ref{cdl} holds. By a similar proof we can show that  (\ref{adl2}) of Definition \ref{cdl} also holds. On the other hand,  (\ref{adl3}) follows from 
\begin{itemize}
\item[ ]$\hspace{0.38cm}  \omega_{X}\circ  (\mu_{A}\ot H)\co (\mu_{A}\otimes \Psi)\co (A\otimes \Psi\otimes A)\circ (A\otimes H\otimes (( \lambda_{A}\otimes A)\circ \delta_{A}))$
	
\item[ ]$=\mu_{X}\circ ((\mu_{X}\circ (i_{A}\otimes i_{A}))\otimes \theta_{X})\circ (A\otimes  (\omega_{X}^{-1}\circ \theta_{X})\otimes A)\circ (A\otimes H\otimes  (( \lambda_{A}\otimes A)\circ \delta_{A}))$ {\scriptsize ({\blue by the  } }
\item[ ]$\hspace{0.38cm}$ {\scriptsize {\blue condition of monoid morphism for $i_{H}$ and  (\ref{factq2}))} }
	
\item[ ]$= \mu_{X}\circ ( (\mu_{X}\circ ((\mu_{X}\circ (X\otimes i_{A}))\otimes i_{H}))\otimes X)\circ (i_{A}\otimes  (\omega_{X}^{-1}\circ \theta_{X})\otimes X)\circ (A\otimes H\otimes  (( \lambda_{A}\otimes i_{A})\circ \delta_{A}))$  
\item[ ]$\hspace{0.38cm}${\scriptsize  ({\blue by (\ref{factq41})})}
	
\item[ ]$= \mu_{X}\circ ( \mu_{X}\otimes X)\circ (i_{A}\otimes  (\mu_{X}\circ (i_{H}\otimes X))\otimes X)\circ (A\otimes H\otimes  (((i_{A}\circ  \lambda_{A})\otimes i_{A})\circ \delta_{A}))$ {\scriptsize  ({\blue by (\ref{factq2})})}
	
\item[ ]$=  \mu_{X}\circ ( \mu_{X}\otimes X)\circ (\omega_{X}\otimes (((\lambda_{X}\circ  i_{A})\otimes i_{A})\circ \delta_{A}))  $ {\scriptsize ({\blue by (\ref{factq1}) and (\ref{antipode-morphism}) for $i_{A}$})} 
	
\item[ ]$= \mu_{X}\circ ( \mu_{X}\otimes X)\circ (\omega_{X}\otimes ((\lambda_{X}\otimes X)\circ \delta_{X}\circ i_{A}))  $  {\scriptsize  ({\blue by the condition of comonoid morphism for $i_{A}$})}
	
\item[ ]$= \omega_{X}\otimes (\varepsilon_{X}\circ i_{A})  $ {\scriptsize  ({\blue by (\ref{rH}) for $X$})}
	
\item[ ]$= \omega_{X}\otimes \varepsilon_{A} $ {\scriptsize  ({\blue by the condition of comonoid morphism for $i_{A}$}).}
	
\end{itemize}
and, similarly, we obtain (\ref{adl4}). 
\end{proof}

As a consequence of the previous theorem we obtain the following result.

\begin{theorem}
\label{prinprev1}
Let $H$, $A$ be Hopf subquasigroups of  a Hopf quasigroup $X$ such that the antipodes of $H$ and $A$ are isomorphisms. Assume that $X$ factorises as $X=AH$. Then, $\omega_{X}$  is an isomorphism of Hopf quasigrous between the wreath product $A\otimes_{\Psi} H$ and $X$, where $\Psi$ is the $a$-comonoidal distributive law defined in the previous theorem.
\end{theorem}

\begin{proof} To complete the proof we only need to show that $\omega_{X}$  is an isomorphism of unital magmas. Indeed, trivially $\omega_{X}\circ \eta_{A\otimes_{\Psi} H}=\eta_{X}$. On the other hand, 
\begin{itemize}
\item[ ]$\hspace{0.38cm} \mu_{X}\circ (\omega_{X}\otimes \omega_{X}) $
	
\item[ ]$= \mu_{X}\circ (i_{A}\otimes (\mu_{X}\circ (i_{H}\otimes \omega_{X})))$ {\scriptsize  ({\blue by (\ref{factq1})})}
	
\item[ ]$=\mu_{X}\circ (i_{A}\otimes (\mu_{X}\circ (\theta_{X}\otimes i_{H})))   $ {\scriptsize ({\blue by (\ref{factq2})})} 
	
\item[ ]$=\mu_{X}\circ (i_{A}\otimes (\mu_{X}\circ (i_{A}\otimes (\mu_{X}\circ (i_{H}\otimes i_{H}))))\circ (A\otimes \Psi\otimes H)    $  {\scriptsize  ({\blue by (\ref{factq1})})}

\item[ ]$= \mu_{X}\circ (i_{A}\otimes  (\omega_{X}\circ (A\otimes \mu_{H}))) \circ (A\otimes \Psi\otimes H)  $  {\scriptsize  ({\blue by the condition of monoid morphism for $i_{H}$})}

\item[ ]$=\omega _{X}\circ \mu_{A\otimes_{\Psi} H} $  {\scriptsize  ({\blue by the condition of monoid morphism for $i_{A}$ and (\ref{factq2})})}
	
\end{itemize}
\end{proof}

In the next theorem we will prove that, in the same way as in the case of Hopf algebras, every factorization of Hopf  quasigroups comes from a matched pair of Hopf quasigroups.

\begin{theorem}
\label{prin} Let $H$, $A$, $X$ be Hopf quasigroups such that the antipodes of $H$ and $A$ are isomorphisms. If $X$ factorizes as $X=AH$, there exists a matched pair of Hopf quasigroups $(A,H)$ such that $X$ is isomorphic to $A\bowtie H$ as Hopf quasigroups.
\end{theorem}

\begin{proof} Let $\Psi$ be the morphism defined in Theorem \ref{prinprev}. Define the actions by 
$$\varphi_{A}=(A\otimes \varepsilon_{H})\circ \Psi, \; \;\;\; \phi_{H}=(\varepsilon_{A}\otimes H)\circ \Psi.$$ 

Then, $(A,H)$ is a matched pair of Hopf quasigroups. Indeed, $(A, \varphi_{A})$ is a left $H$-module because, by (\ref{dl4}) and (\ref{eta-eps}) for $H$,  we have 
$$\varphi_{A}\circ (\eta_{H}\otimes A)=(A\otimes \varepsilon_{H})\circ \Psi\circ (\eta_{H}\otimes A)=A\otimes (\varepsilon_{H}\circ \eta_{H})=id_{A}, $$
and, by (\ref{mu-eps}) for $H$ and (\ref{dl2-1}), 
$$\varphi_{A}\circ (H\otimes \varphi_{A})=(A\ot (\varepsilon_{H}\circ \mu_{H}))\co ( \Psi\ot H)\co (H\ot \Psi)=\varphi_{A}\co (\mu_{H}\ot A).$$

Also, using that $\psi$ is a comonoid morphism, the naturality of $c$ and the properties of the counits, we have that $(A, \varphi_{A})$ is a left $H$-module comonoid. Similarly we can prove that $(H, \phi_{H})$ is a right $H$-module comonoid.

On the other hand, (\ref{d1}) follows from (\ref{dl3}) and (\ref{d2})  from (\ref{dl4}). The identity (\ref{d3}) is a consequence of the condition of comonoid morphism of $\Psi$ and the properties of the counits. On the other hand, (\ref{d4}) follows by (\ref{dl1}) (or (\ref{dl1-1}) because the antipodes are isomorphisms). The equality (\ref{d5}) follows by (\ref{adl1}) and (\ref{d6}) can be proved thanks to (\ref{adl2}). Similarly to (\ref{d4}), (\ref{d7}) follows by (\ref{dl2}) (or (\ref{dl2-1}) because the antipodes are isomorphims). Finally, (\ref{d8}) is a consequence of (\ref{adl4}) and (\ref{d7}) follows by (\ref{adl3}).

Note that, by the condition of comonoid morphism of $\Psi$, the naturality of $c$ and the properties of the units,  we have that 
$$\Psi=(\varphi_{A}\otimes \phi_{H})\circ \delta_{H\otimes A}.$$

Therefore, $A\otimes_{\Psi}H=A\bowtie H$ and, by Theorem \ref{prinprev1}, $X$ is isomorphic to $A\bowtie H$ as Hopf quasigroups.
\end{proof}

\begin{theorem}
\label{prin-cor} Let $H$, $A$, $X$ be Hopf quasigroups such that the antipodes of $H$ and $A$ are isomorphisms. Then, $X$ factorizes as $X=AH$ iff there exists a matched pair of Hopf quasigroups $(A,H)$ such that $X$ is isomorphic to $A\bowtie H$ as Hopf quasigroups.
\end{theorem}

\begin{proof}
The proof follows by Theorem \ref{prin} and Example \ref{ex1fact}.
\end{proof}

Of course, the definitions and results of this section admit a dual version in the Hopf coquasigroup setting. We will state everything below, omitting the proofs, since these follow by duality from those done in the context of Hopf quasigroups.

\begin{definition}
\label{cfactq}
{\rm Let $Y$ be a Hopf coquasigroup in {\sf C}. Let $D$, $B$ be Hopf coquasigroups such that there exist Hopf coquasigroup epimorphisms  $p_{D}:Y\rightarrow D$ and   $p_{B}:Y\rightarrow B$. Let  $u_{Y}$ and $v_{Y}$ be the morphisms defined by 
$$u_{Y}=(p_{D}\otimes p_{B})\circ \delta_{Y}:Y\rightarrow D\otimes B, \;\;\;v_{Y}=(p_{B}\otimes p_{D})\circ \delta_{Y}:Y\rightarrow B\otimes D.$$
		
We will say that  $Y$ cofactorizes as $Y=D\bullet B$ if $u_{Y}$ is an isomorphism and the following identities
\begin{equation}\label{cfactq1}
(Y\otimes u_{Y})\circ \delta_{Y}=(((Y\otimes p_{D})\circ\delta_{Y})\otimes p_{B})\circ \delta_{Y},
\end{equation}
\begin{equation}\label{cfactq2}
(u_{Y}\otimes Y)\circ \delta_{Y}=(p_{D}\otimes ((p_{B}\otimes Y)\circ\delta_{Y}))\circ \delta_{Y},
\end{equation}
\begin{equation}\label{cfactq3}
(Y\otimes ((\lambda_{B}\otimes \lambda_{D})\circ v_{Y}))\circ \delta_{Y}=(((Y\otimes (\lambda_{B}\circ p_{B}))\circ\delta_{Y})\otimes (\lambda_{D}\circ p_{D}))\circ \delta_{Y},
\end{equation}
\begin{equation}\label{cfactq4}
(((\lambda_{B}\otimes \lambda_{D})\circ v_{Y})\otimes Y)\circ \delta_{Y}=((\lambda_{B}\circ p_{B})\otimes (((\lambda_{D}\circ p_{D})\otimes Y)\circ\delta_{Y}))\circ \delta_{Y}
\end{equation}
hold.
		
Note that, if $D$ and $B$ are finite we can remove the antipodes in (\ref{cfactq3}) and (\ref{cfactq4}). Then these identities became  in 
\begin{equation}\label{cfactq31}
(Y\otimes  v_{Y})\circ \delta_{Y}=(((Y\otimes  p_{B})\circ\delta_{Y})\otimes p_{D})\circ \delta_{Y},
\end{equation}
\begin{equation}\label{cfactq41}
(v_{Y}\otimes Y)\circ \delta_{Y}=( p_{B}\otimes ((p_{D}\otimes Y)\circ\delta_{Y}))\circ \delta_{Y}.
\end{equation}
		
Finally, note that $u_{Y}$ and $v_{Y}$ are monoid morphisms because $p_{D}$, $p_{B}$ are monoid morphisms and (\ref{mu-eps}) and (\ref{delta-mu}) holds for the Hopf coquasigroup $Y$. Obviously if $u_{Y}$ is a monoid isomorphism, its dual is a comonoid isomorphism.
}
\end{definition}

\begin{example}
{\rm 
\label{ex1factc}  Suppose that $(B, D)$ is a comatched pair of Hopf coquasigroups.
By Example \ref{1c},   the double cross coproduct $D\infty B$ is a Hopf coquasigroup.
The morphisms $p_{B}=\varepsilon_{D}\otimes B:D\infty B\rightarrow B$ and $p_{D}=D\otimes \varepsilon_{B}:D\infty B\rightarrow D$ are morphisms of Hopf coquasigroups such that 
$$u_{D\infty B}=id_{D\otimes B},\;\;\; v_{D\infty B}=\Omega$$
and the equalities (\ref{cfactq1})-(\ref{cfactq4}) hold. Therefore, any double cross coproduct of Hopf coquasigroups induces an example of cofactorization. 
}
\end{example}

\begin{theorem}
\label{cprinprev} Let $D$, $B$ and $Y$ be Hopf coquasigroups such that there exist Hopf coquasigroup epimorphisms  $p_{D}:Y\rightarrow D$ and   $p_{B}:Y\rightarrow B$. If  $Y$ cofactorises as $Y=D\bullet B$, the morphism 
$$\Omega=v_{Y}\circ u_{Y}^{-1}:D\otimes B\rightarrow B\otimes D$$
is a monoidal codistributive law of $D$ over $B$. Moreover, if the antipodes of $D$ and $B$ are isomorphisms $\Omega$ is an $a$-monoidal codistributive law of $D$ over $B$.
\end{theorem}

\begin{theorem}
\label{cprinprev1}
Let $D$, $B$ and $Y$ be Hopf coquasigroups such that there exist Hopf coquasigroup epimorphisms  $p_{D}:Y\rightarrow D$, $p_{B}:Y\rightarrow B$ and suppose that the antipodes of $D$ and $B$ are isomorphisms. Assume that $Y$ cofactorises as $Y=D\bullet B$. Then, $u_{Y}$  is an isomorphism of Hopf coquasigrous between $Y$ and the  the wreath coproduct $D\otimes_{\Omega} B$, where $\Omega$ is the $a$-monoidal codistributive law defined in the previous theorem.
\end{theorem}

\begin{theorem}
\label{cprin} Let $D$, $B$, $Y$ be Hopf coquasigroups such that the antipodes of $D$ and $B$ are isomorphisms. If $Y$ cofactorises as $Y=D\bullet B$, there exists a comatched pair of Hopf coquasigroups $(B,D)$ such that $Y$ is isomorphic to $D\infty B$ as Hopf coquasigroups.
\end{theorem}

\begin{theorem}
\label{cprin-cor} Let $D$, $B$, $Y$ be Hopf coquasigroups such that the antipodes of $D$ and $B$ are isomorphisms. Then, $Y$ cofactorises as $Y=D\bullet B$ iff there exists a comatched pair of Hopf coquasigroups $(B,D)$ such that $Y$ is isomorphic to $D\infty B$ as Hopf coquasigroups.
\end{theorem}

\begin{example}
{\rm  Let $S_{3}=\{\sigma_{0}, \sigma_{1}, \sigma_{2}, \sigma_{3}, \sigma_{4}, \sigma_{5}\}$ be the nonabelian group where $\sigma_{0}$ is  the identity, $o(\sigma_{1})=o(\sigma_{2})=o(\sigma_{3})=2$ and $o(\sigma_{4})=o(\sigma_{5})=3$. Let $u$ be an additional element such that $u^2=1$. Then, by  \cite[Theorem 1]{Chein}, the set 
$$L=M(S_{3},2)=\{\sigma_{i}u^{\alpha}\;; \; \alpha=0,1\}$$
is an I.P. loop where the product is defined by 
$$\sigma_{i}u^{\alpha}\bullet \;\sigma_{j}u^{\beta}=(\sigma_{i}^{\nu}\sigma_{j}^{\mu})^{\nu}u^{\alpha +\beta},\;\;\; \nu=(-1)^{\beta}, \; \mu=(-1)^{\alpha +\beta}$$
and the inverse by 
$$(\sigma_{i}u^{\alpha})^{-1}=\sigma_{i}^{(-1)^{\alpha +1}}u^{\alpha}.$$

Let ${\Bbb F}$ be a field such that Char(${\Bbb F}$)$\neq 2$ and denote the tensor product
over ${\Bbb F}$ as $\ot$. By Examples \ref{exaHq}, we have that $A={\Bbb F}L$ is a cocommutative Hopf quasigroup. 

On the other hand, let $H_{4}$ be the $4$-dimensional Taft Hopf algebra. This Hopf algebra is the smallest non-commutative, non-cocommutative Hopf algebra. The basis of $H_{4}$ is $\{1,x,y,w=xy\}$ and the multiplication table is defined by 
\begin{center}
\begin{tabular}{|c|c|c|c|c|}
\hline  $\;$ & $x$ & $y$ & $w$   \\
\hline  $ x$ &  $1$ & $w$ & $y$ \\
\hline  $ y$ &  $-w$ & $0$ & $0$ \\
\hline  $ w$ &  $-y$ & $0$ & $0$ \\
\hline
\end{tabular}
\end{center}

The coproduct of $H_{4}$ is given by 
$$ \delta_{H_{4}}(x)=x\ot x,\; \delta_{H_{4}}(y)=y\ot x +1\ot y,\; \delta_{H_{4}}(w)=w\ot 1 +x\ot w,$$
$$\varepsilon_{H_{4}}(x)=1_{\Bbb F},\; \varepsilon_{H_{4}}(y)=\varepsilon_{H_{4}}(w)=0, $$
and the antipode $\lambda_{H_{4}}$ is described by 
$$\l \lambda_{H_{4}}(x)=x,\; \lambda_{H_{4}}(y)=w,\; \lambda_{H_{4}}(w)=-y.$$

Following \cite[Example 4.12]{our2} the morphism $\tau:A\ot H_{4}\rightarrow {\Bbb F}$, defined by 
$$\tau (\sigma_{i}u^{\alpha}\ot z)=\left\{ \begin{array}{ccc} 1 & {\rm if} &
z=1 \\
(-1)^{\alpha} & {\rm if} &
z=x \\ 
0 & {\rm if} &
z=y, w 
\end{array}\right.$$
is a skew pairing such that $\tau=\tau^{-1}$.  Then, by  \cite[Proposition 5.2]{our2}, $(A, \varphi_{A})$, where 
$$\varphi_A=(\tau\ot A\ot \tau)\co (A\ot H_{4}\ot \delta_{A}\ot H_{4})\co \delta_{A\ot H_{4}}\co c_{H_{4},A},$$ 
is a left $H_{4}$-module comonoid  and  $(H_{4}, \phi_{H_{4}})$, where 
$$\phi_{H_{4}}=(\tau\ot H_{4}\ot \tau)\co (A\ot H_{4}\ot c_{A,H_{4}}\ot H_{4})\co (A\ot H_{4}\ot A\ot \delta_{H_{4}})\co \delta_{A\ot H_{4}}\co c_{H_{4},A},$$ 
is a right  $H$-module comonoid. Moreover, by \cite[Corollary 5.6]{our2}, the pair $(A, H_{4})$ is a matched pair of Hopf quasigroups.  More concretely, 
$$\varphi_A (1\otimes \sigma_{i}u^{\alpha} )=\sigma_{i}u^{\alpha}, \;\; \varphi_A (x\otimes \sigma_{i}u^{\alpha} )=\sigma_{i}u^{\alpha},\;\; \varphi_A (y\otimes \sigma_{i}u^{\alpha} )=\varphi_A (w\otimes \sigma_{i}u^{\alpha} )=0,$$
and 
$$\phi_{H_{4}}(1\otimes \sigma_{i}u^{\alpha} )=1, \;\;\phi_{H_{4}}(x\otimes \sigma_{i}u^{\alpha} )=x, \;\;\phi_{H_{4}}(y\otimes \sigma_{i}u^{\alpha} )=(-1)^{\alpha}y, \;\; \phi_{H_{4}}(w\otimes \sigma_{i}u^{\alpha} )=(-1)^{\alpha}w.$$

Then, 
$$\Psi (1\otimes \sigma_{i}u^{\alpha} )=\sigma_{i}u^{\alpha}\ot 1, \;\;\Psi (x\otimes \sigma_{i}u^{\alpha} )=\sigma_{i}u^{\alpha}\ot x, $$
$$\;\;\Psi(y\otimes \sigma_{i}u^{\alpha} )=(-1)^{\alpha}\sigma_{i}u^{\alpha}\ot y, \;\; \Psi (w\otimes \sigma_{i}u^{\alpha} )=(-1)^{\alpha}\sigma_{i}u^{\alpha}\ot w.$$

Therefore, the product table of $A\bowtie H_{4}$ is
\begin{center}
\begin{tabular}{|c|c|c|c|c|}
\hline  $\;$ & $\sigma_{j}u^{\beta}\ot 1$ & $\sigma_{j}u^{\beta}\ot x$ & $\sigma_{j}u^{\beta}\ot y$  & $\sigma_{j}u^{\beta}\ot w$\\
\hline  $ \sigma_{j}u^{\alpha}\ot 1$ &  $\sigma_{i}u^{\alpha}\bullet \;\sigma_{j}u^{\beta} \ot 1$ & $\sigma_{i}u^{\alpha}\bullet \;\sigma_{j}u^{\beta} \ot x$ & $\sigma_{i}u^{\alpha}\bullet \;\sigma_{j}u^{\beta} \ot y$ & $\sigma_{i}u^{\alpha}\bullet \;\sigma_{j}u^{\beta} \ot w$\\
\hline  $ \sigma_{j}u^{\alpha}\ot x$ &  $\sigma_{i}u^{\alpha}\bullet \;\sigma_{j}u^{\beta} \ot x$ & $\sigma_{i}u^{\alpha}\bullet \;\sigma_{j}u^{\beta} \ot 1$ & $\sigma_{i}u^{\alpha}\bullet \;\sigma_{j}u^{\beta} \ot w$ & $\sigma_{i}u^{\alpha}\bullet \;\sigma_{j}u^{\beta} \ot y$\\
\hline  $ \sigma_{j}u^{\alpha}\ot y$ &  $(-1)^{\beta}\sigma_{i}u^{\alpha}\bullet \;\sigma_{j}u^{\beta} \ot y$ & $(-1)^{\beta+1}\sigma_{i}u^{\alpha}\bullet \;\sigma_{j}u^{\beta} \ot w$ & $0$ & $0$\\
\hline  $ \sigma_{j}u^{\alpha}\ot w$ &  $(-1)^{\beta}\sigma_{i}u^{\alpha}\bullet \;\sigma_{j}u^{\beta} \ot w$ & $(-1)^{\beta+1}\sigma_{i}u^{\alpha}\bullet \;\sigma_{j}u^{\beta} \ot y$ & $0$ & $0$\\
\hline
\end{tabular}
\end{center}
$$ $$
and the antipode is given by:
$$\lambda_{A\bowtie H_{4}}(\sigma_{i}u^{\alpha}\ot 1)=\sigma_{i}^{(-1)^{\alpha +1}}u^{\alpha}\ot 1, \;\; \lambda_{A\bowtie H_{4}}(\sigma_{i}u^{\alpha}\ot x)=\sigma_{i}^{(-1)^{\alpha +1}}u^{\alpha}\ot x, $$
$$\lambda_{A\bowtie H_{4}}(\sigma_{i}u^{\alpha}\ot y)=(-1)^{\alpha}\sigma_{i}^{(-1)^{\alpha +1}}u^{\alpha}\ot w, \;\; \lambda_{A\bowtie H_{4}}(\sigma_{i}u^{\alpha}\ot x)=(-1)^{\alpha+1}\sigma_{i}^{(-1)^{\alpha +1}}u^{\alpha}\ot y.$$

Then, $A\bowtie H_{4}$ is an example of non-commutative, non-cocommutative Hopf quasigroup and its dual $(A\bowtie H_{4})^{\ast}=(H_{4})^{\ast}\infty A^{\ast}$ is an example of non-commutative, non-cocommutative Hopf coquasigroup.
}
\end{example}

\section*{Acknowledgements}
The  author was supported by  Ministerio de Ciencia e Innovaci\'on of Spain. Agencia Estatal de Investigaci\'on. Uni\'on Europea - Fondo Europeo de Desarrollo Regional (FEDER). Grant PID2020-115155GB-I00: Homolog\'{\i}a, homotop\'{\i}a e invariantes categ\'oricos en grupos y \'algebras no asociativas.

\end{document}